\documentclass[english,11pt]{amsart}

\usepackage{amsrefs}

\usepackage{amssymb,amsmath,amsfonts,amsthm}
\usepackage{graphicx}

\usepackage{graphicx,epic,eepic}
\usepackage{epsfig}
\usepackage{latexsym,esint}
\usepackage{color}

\newtheorem{theorem}{Theorem}[section]
\newtheorem{corollary}[theorem]{Corollary}
\newtheorem{prop}[theorem]{Proposition}
\newtheorem{lemma}[theorem]{Lemma}
\newtheorem{remark}[theorem]{Remark}
\newtheorem{definition}[theorem]{Definition}

\theoremstyle{remark}
\newtheorem{rmk}[theorem]{Remark}

\numberwithin{equation}{section}

\def\neweq#1{\begin{equation}\label{#1}}
\def\endeq{\end{equation}}

\def\ri{\rightarrow}

\def\vphi{\varphi}

\def\a{\alpha}
\def\b{\beta}

\def\eq#1{(\ref{#1})}

\def\R{{\mathbb R} }

\def\rmc{\R^N\setminus\{0\}}

\def\o{\Omega }

\begin{document}

\title{A new critical curve for the Lane-Emden system}

\author[W. Chen]{Wenjing Chen}

\address{Departamento de Ingenier\'ia Matem\'atica and CMM, Universidad de Chile, Casilla 170 Correo 3, Santiago, Chile}

\email{wjchen1102@yahoo.com.cn}

\author[L. Dupaigne]{Louis Dupaigne}

\address{LAMFA, UMR CNRS 7352, Universit\'e de Picardie Jules Verne, 33 rue St Leu, 80039, Amiens Cedex, France}

\email{louis.dupaigne@math.cnrs.fr}

\author[M. Ghergu]{Marius Ghergu}

\address{School of Mathematical Sciences, University College Dublin, Belfield, Dublin 4, Ireland}

\email{marius.ghergu@ucd.ie}

\thanks{}


\subjclass[2000]{}



\maketitle

\begin{abstract}
We study stable positive radially symmetric solutions for the Lane-Emden system $-\Delta u=v^p$ in $\R^N$, $-\Delta v=u^q$ in $\R^N$, where $p,q\geq 1$. We obtain a new critical hyperbola that optimally describes the existence of such solutions.

\end{abstract}

\section{Introduction}
We consider the Lane-Emden system
\begin{equation}\label{le}
\left\{\begin{aligned}
-&\Delta u=v^p, \; u>0&& \quad\mbox{ in }\R^N,\\
-&\Delta v=u^q, \; v>0&& \quad\mbox{ in }\R^N,
\end{aligned}\right.
\end{equation}
where $N\ge1$ and $p\ge q>0$.
Introduced independently by Mitidieri \cite{mitidieri} and Van der Vorst \cite{vandervorst}, the Sobolev critical hyperbola plays a crucial role in the analysis of \eqref{le}. In particular, Mitidieri \cite{mit96} (see also Serrin and Zou \cite{serrin-zou}) proved that \eq{le} has a nontrivial radially symmetric solution if and only if $(p,q)$ lies on or above the hyperbola i.e. when
\neweq{le1}
\frac{1}{p+1}+\frac{1}{q+1}\le 1-\frac{2}{N}.
\endeq
The Lane-Emden conjecture states that such a result should continue to hold for any positive solution (not necessarily radially symmetric). See Souplet \cite{souplet} and the references therein for the progress on this conjecture.

In this paper we characterize the stability of radially symmetric solutions of the Lane-Emden system \eq{le}, the definition of which we recall now. 
\begin{definition}
A solution $(u,v)$ to \eq{le} is stable if there exists a positive supersolution of the linearized system i.e. if there exists $(\phi, \psi)\in C^2(\R^{N})^{2}$ such that
$$
\left\{\begin{aligned}
-&\Delta \phi\geq pv^{p-1}\psi \; && \quad\mbox{ in }\R^{N},\\
-&\Delta \psi\geq  q u^{q-1}\phi&& \quad\mbox{ in }\R^{N},\\
&\phi,\psi>0&& \quad\mbox{ in }\R^{{N}}.
\end{aligned}\right.
$$
\end{definition}
Let us also recall  that if \eq{le1} holds, then
\neweq{lesingular}
(u_s,v_s)=(a|x|^{-\alpha},b|x|^{-\beta}), \quad x\in \R^N\setminus \{0\}
\endeq
is a weak solution of \eq{le} provided
\neweq{leab}
\alpha=\frac{2(p+1)}{pq-1}\;,\quad \beta=\frac{2(q+1)}{pq-1}
\endeq
and
$
a=(ST^p)^{\frac{1}{pq-1}}\;, b=(S^qT)^{\frac{1}{pq-1}},
S=\alpha(N-2-\alpha)\;, T=\beta(N-2-\beta).
$


Our main result states that the stability of a radial solution of the Lane-Emden system is determined by the position of the exponents $(p,q)$ with respect to a new critical curve, which we christen ``Joseph and Lundgren'', since the exponent introduced by these authors in \cite{jl72}  is the intersection of the curve with the diagonal $p=q$ .
\begin{theorem}\label{stable}
Assume $p\ge q\ge1$.
\begin{enumerate}
\item[(i)] If $N\ge11$ and $(p,q)$ lies on or above the Joseph-Lundgren critical curve i.e.
\neweq{lecv}
\Big[\frac{(N-2)^2-(\a-\b)^2}{4}  \Big]^2\geq pq\a \b (N-2-\a)(N-2-\b),
\endeq
then any radially symmetric solution $(u,v)$ of \eq{le} is stable and satisfies
$$
u<u_s\quad\mbox{ and }\quad v<v_s\quad\mbox{ in }\;\R^N\setminus\{0\},
$$
where $(u_{s},v_{s})$ is the singular solution given by \eqref{lesingular} and $\alpha,\beta$ are the scaling exponents given by \eqref{leab}.
\item[(ii)] If $N\le 10$ or if $N\ge11$ and \eq{lecv} fails, then there is no stable radially symmetric solution of \eqref{le}.
\end{enumerate}
\end{theorem}

\begin{rmk}Equation \eqref{lecv} is derived by studying the stability of the singular solution $(u_{s},v_{s})$  given by \eqref{lesingular}.
\end{rmk}

\begin{figure}[tbph]
\begin{center}
\includegraphics[height=9cm]{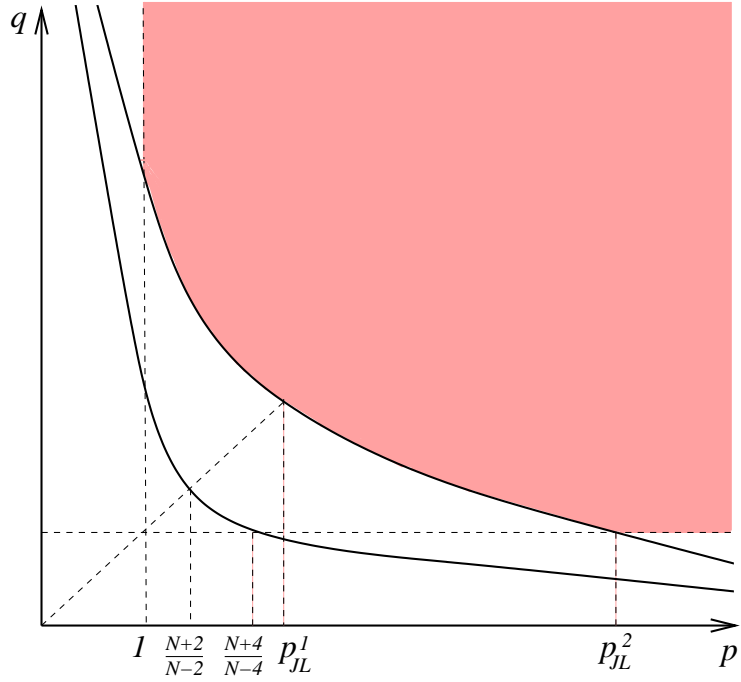}
\end{center}
\caption{The stable region (shaded) for radially symmetric solutions of the Lane-Emden system \eq{le}.}
\end{figure}

\begin{rmk} \

\begin{itemize}
\item The above theorem was first proved by Cowan for $1\le N\le 10$, $p\ge q\ge 2$ and $(u,v)$ not necessarily radial. See \cite{cowan}.
\item In the case $p=q$, using Remarks 1.1(a) and 2.1(a) of Souplet \cite{souplet} and Farina's seminal work for the case of a single equation \cite{farina},  part (ii) of the theorem readily follows. The result continues to hold for possibly nonradial solutions, assumed to be stable only outside a compact set. 
\item In the biharmonic case $q=1$, the theorem was first proved by Karageorgis \cite{karageorgis} using the asymptotics found by Gazzola and Grunau in \cite{gg}.
\item In all the other cases, only partial results were known. To the authors knowledge, the state of the art for nonradial solutions is contained in the following references: Wei and D. Ye \cite{wy2}, Wei, Xu and Yang \cite{wxx}, Hajlaoui, A. Harrabi and D. Ye \cite{hhy} for the biharmonic case, and Cowan \cite{cowan} for the general case. We believe that the methods of the paper \cite{dggw} by Goubet, Warnault and two of the authors should slightly improve the known results (and coincide with \cite{hhy} in the biharmonic case).
\item Our result does not cover the case where one of the exponents is less than 1.
\item The left hand-side in \eq{lecv} is related to the following Hardy-Rellich inequality :
\neweq{mussina}
\int_{\R^N}|x|^{2-\gamma}|\Delta \vphi|^2 dx\geq C_{\gamma}\int_{\R^N} |x|^{-2-\gamma} \vphi^2 dx.
\endeq
The optimal constant $C_\gamma$ in the class of radially symmetric functions $\vphi=\vphi(|x|)$ is given by
\neweq{optim}
C_\gamma=\inf_{\substack{\varphi\in C^\infty_c(\R^N\setminus\{0\})\\  0\neq \varphi=\vphi(|x|)}}\frac{\displaystyle \int_{\o}|x|^{2-\gamma}|\Delta \varphi|^2dx}{\displaystyle\int_{\o}|x|^{-2-\gamma}\vphi^2 dx}=\Big[\frac{(N-2)^2-\gamma^2}{4}  \Big]^2,
\endeq
and the above infimum is never achieved. See Caldiroli and Musina \cite{cmus}. We remark that the optimal constant $C_\gamma$ in \eq{optim} corresponds to the left hand-side in \eq{lecv} with $\gamma=\a-\b\in [0,2)$.

\end{itemize}
\end{rmk}

As an immediate corollary of Theorem \ref{stable} and standard blow-up analysis, we obtain the following regularity result.

\begin{corollary}Let $B$ denote the unit ball of $\R^{N}$, $N\ge1$, $\lambda,\mu>0$. Let $f,g\in C^{1}(\R)$ be two nondecreasing functions such that $f(0)\ge0$, $g(0)>0$, $f'(0)g'(0)>0$ and
$$
\lim_{t\to+\infty} \frac{f'(t)}{t^{p-1}}=a,\qquad \lim_{t\to+\infty} \frac{g'(t)}{t^{q-1}}=b
$$
for some $a,b>0$, $p\ge q\ge 1$, $pq>1$. Then, any extremal solution to the system
\begin{equation}\label{x}
\left\{\begin{aligned}
-&\Delta u=\lambda f(v), \; u>0&& \quad\mbox{ in }B,\\
-&\Delta v=\mu g(u), \; v>0&& \quad\mbox{ in }B,\\
&u=v=0&&\quad\mbox{ on }\partial B
\end{aligned}\right.
\end{equation}
is bounded if either $N\le 10$ or if $N\ge11$ and $(p,q)$ lies below the Joseph-Lundgren critical curve i.e.\eqref{lecv} fails.
\end{corollary}
For the notion of extremal solution for systems, we refer to Montenegro \cite{montenegro}. See also Cowan \cite{cowan2} for partial results on general domains.
The proof is a straightforward adaptation of  Theorem 1.8 in \cite{ddf}, using the version of the blow-up technique introduced by Polacik, Quittner and Souplet \cite{pqs}, so we skip it.

\section{Preliminary Results}

The following three results will serve for the purpose of comparing solutions. In the lemma below, we say that a solution is strictly stable in a bounded region $\Omega\subset\R^{N}$ if the principal eigenvalue of the linearized equation with Dirichlet boundary conditions in $\Omega$ is strictly positive.
\begin{lemma} \label{corst}
Let $(u,v)\in C^{2}(\R^{N})^2$ be a stable solution of \eq{le}. 
Then, given any bounded domain $\o\subset\R^{N}$, $(u,v)$ is strictly stable in $\o$. In particular, the linearized operator satisfies the maximum principle, that is,  any pair $(\phi, \psi)\in C^{2}(\overline{\Omega})^2$ such that
$$
\left\{\begin{aligned}
-&\Delta \phi\geq pv^{p-1}\psi \; && \quad\mbox{ in }\o,\\
-&\Delta \psi\geq qu^{q-1}\phi\;&& \quad\mbox{ in }\o,\\
&\phi,\psi\geq 0&& \quad\mbox{ on }\partial \o,
\end{aligned}\right.
$$
satisfies $\phi, \psi\geq 0$ in $\o$.
\end{lemma}
\proof Since $(u,v)$ is stable in $\R^{N}$, the linearized equation has a strict supersolution in $\Omega$.
As observed by Sweers \cite{sweers92} and Busca-Sirakov \cite{bs04}, this implies in turn that the principal eigenvalue of the linearized operator with Dirichlet boundary conditions in $\Omega$ is strictly positive and equivalently that the maximum principle holds.
\hfill\qed

\

In the next lemma, we say that a solution is minimal if it lies below any (local) supersolution of the same equation. See e.g. \cite{ddgr} for the notion of minimal solution.
\begin{lemma}\label{stablebounded}
Assume $p\ge q\ge1$ and let $\o\subset\R^{N}$ be a bounded domain, $a,b\in C(\partial \o)$, $a,b\geq 0$.  If $(u,v)\in C^{2}(\overline{\Omega})^{2}$ is a strictly stable solution of
\neweq{lecball}
\left\{\begin{aligned}
-&\Delta u=v^p \; && \quad\mbox{ in }\o,\\
-&\Delta v=u^q&& \quad\mbox{ in }\o,\\
&u=a(x)\,,v=b(x)&& \quad\mbox{ on }\partial \o,
\end{aligned}\right.
\endeq
then $(u,v)$ is minimal.
\end{lemma}
\begin{proof}
Assume that $(u,v)$ is a strictly stable solution of \eqref{lecball}.
By the maximum principle, 
$$
u\geq \min_{\partial\o}a\,,\quad v\geq \min_{\partial\o}b\quad\mbox{ in } \o.
$$ 
In particular, there exists the minimal solution $(u_{m},v_{m})$ of \eq{lecball} and 
$$
u\geq u_m\geq \min_{\partial\o}a\,,\quad v\geq v_m\geq \min_{\partial\o}b\quad\mbox{ in } \o.
$$ 
Set $\phi=u-u_m$, $\psi=v-v_m$. Then, $\phi, \psi\geq 0$ in $\o$ and, since $p\ge q\ge1$,
$$
\left\{\begin{aligned}
-&\Delta \phi=v^p-v^p_m\leq pv^{p-1}\psi \; && \quad\mbox{ in }\o,\\
-&\Delta \psi=u^q-u_m^q\leq q u^{q-1}\phi&& \quad\mbox{ in }\o,\\
&\phi=\psi=0&& \quad\mbox{ on }\partial \o.
\end{aligned}\right.
$$
Since $(u,v)$ is strictly stable, the maximum principle holds and implies that $\phi, \psi\leq 0$ in $\o$.  It follows that $\phi\equiv\psi\equiv 0$, that is, $u=u_m$ and $v=v_m$.
\end{proof}
As an immediate consequence of the two previous lemmas, we obtain
\begin{corollary}\label{lecbalcor}
Let $(u,v)\in C^{2}(\R^{N})^2$ be a stable solution of \eq{le} and let $(u_s,v_s)$ be the singular solution defined by  \eq{lesingular}. If there exists $R>0$ such that $u(R)\leq u_s(R)$ and $v(R)\leq v_s(R)$, then
$$
u<u_s\quad\mbox{ and }\quad v<v_s\quad\mbox{ in }B_R\setminus\{0\}.
$$
\end{corollary}
\begin{proof}
Since $u_s(0)=v_s(0)=\infty$, there exists $r\in (0,R)$ such that
\neweq{finq}
u<u_s\quad\mbox{ and }\quad v<v_s\quad\mbox{ in }\overline{B_r}\setminus\{0\}.
\endeq
We next apply  Lemma \ref{stablebounded} for $\o=B_R\setminus\overline{B_r}$, $a(x)=u$, $b(x)=v$. Thus $(u,v)$ is the minimal solution of \eq{lecball} and $u<u_s$, $v<v_s$ in $B_R\setminus\overline{B_r}$. This last inequality together with \eq{finq} yield the conclusion.
\end{proof}
%
%
%

\subsection{Stability of the singular solution.} In this part we investigate the stability of the singular solution $(u_s,v_s)$ given by \eq{lesingular}.

\begin{prop}\label{stabsing}
The following are equivalent:
\begin{enumerate}
\item[(i)] The singular solution $(u_s,v_s)$ is stable in $\R^N\setminus\{0\}$;
\item[(ii)] The singular solution $(u_s,v_s)$ is stable outside of some compact set;
\item[(iii)] $(p,q)$ satisfies \eq{lecv}.
\end{enumerate}
\end{prop}
\begin{proof}
Since the implication (i)$\Rightarrow$(ii) is trivial, we only need to prove the implications
$$
{\rm (ii)}\Rightarrow{\rm (iii)}\Rightarrow{\rm (i)}
$$

Assume first that (ii) holds, that is, the singular solution
$(u_s,v_s)$ is stable outside of a compact set. Thus, $(u_s,v_s)$ is stable in $\R^N\setminus \overline{B_r}$ for some $r>0$. By scale invariance, $(u_s,v_s)$ is stable in $\R^N\setminus \overline{B_\rho}$ for all $\rho>0$.


Set $\gamma=\alpha-\beta$, where $\alpha,\beta$ are the scaling exponents given by \eqref{leab}  and let $K_{1}, K_{2}$ be the constants such that
$$
 pv_s^{p-1}= K_1|x|^{-2+\gamma}\qquad\text{and}\qquad
qu_s^{q-1} =K_2|x|^{-2-\gamma}.
$$
Then, $(p,q)$ satisfies \eq{lecv} if and only if
$$
C_{\gamma}\ge K_{1}K_{2},
$$
where $C_{\gamma}$ is given by \eqref{optim}. Assume by contradiction that $(p,q)$ does not satisfy \eq{lecv}.
Then, we may find an open annular region $\o=B_{R_1}\setminus \overline{B_{R_2}}$ such that
\neweq{mu}
\lambda:=\min_{\varphi\in H\setminus\{0\}} \frac{\displaystyle \int_{\o}|x|^{2-\gamma}|\Delta \varphi|^2dx}{\displaystyle\int_{\o}|x|^{-2-\gamma}\vphi^2 dx}<K_1K_2,
\endeq
where $H$ is the space of radial functions $\varphi$ such that  $\int_{\o}|x|^{2-\gamma}|\Delta \varphi|^2dx<+\infty$ and $\varphi=0$ on $\partial\o$.
Let $\vphi>0$ be a minimizer of \eq{mu}, so that letting $\psi=\vert x\vert^{2-\gamma}(-\Delta\varphi)$, we have
$$
\left\{\begin{aligned}
-&\Delta \vphi= |x|^{-2+\gamma} \psi, \; \vphi>0&& \quad\mbox{ in }\Omega,\\
-&\Delta \psi=\lambda |x|^{-2-\gamma} \vphi, \; \psi>0&& \quad\mbox{ in }\Omega,\\
&\vphi=\psi=0 && \quad\mbox{ on }\partial\Omega.
\end{aligned}\right.
$$
Since $(u_s,v_s)$ is strictly stable in $\o$, thanks to \cite[Theorem 1.1]{sweers92}, there also exists $(\tilde \vphi, \tilde \psi)\in C^2(\overline\o)^2$ such that
$$
\left\{\begin{aligned}
-&\Delta \tilde\vphi= K_1|x|^{-2+\gamma} \tilde\psi, \; \tilde\vphi>0&& \quad\mbox{ in }\Omega,\\
-&\Delta \tilde\psi=K_2|x|^{-2-\gamma} \tilde\vphi+1, \; \tilde\psi>0&& \quad\mbox{ in }\Omega,\\
&\tilde\vphi=\tilde\psi=0 && \quad\mbox{ on }\partial\Omega.
\end{aligned}\right.
$$
A straightforward integration by part shows that $\vphi$ and $\tilde\vphi$ satisfy
$$
\langle \vphi, \tilde\vphi\rangle:=\int_{\o}|x|^{2-\gamma}\Delta \vphi\Delta \tilde\vphi dx\le0
$$
which is impossible, since both $\psi$ and $\tilde\psi$ are positive. Hence $(p,q)$ satisfies \eq{lecv} and we have proved that (ii) implies (iii).

Assume now (iii). It is easy to see that
\neweq{fipsi}
\phi(x)=\frac{4K_1}{(N-2-\gamma)(N-2+\gamma)}|x|^{-\frac{N-2-\gamma}{2}}\;,\quad
\psi(x)=|x|^{-\frac{N-2+\gamma}{2}}
\endeq
satisfy
\neweq{lec1}
\begin{aligned}
-\Delta \phi&= pv_s^{p-1} \psi\\
-\Delta \psi&\geq qu_s^{q-1} \phi
\end{aligned}
\endeq
in $\R^N\setminus\{0\}$, which means that $(u_s,v_s)$ is stable in $\R^N\setminus\{0\}$.
\end{proof}

\section{Proof of Theorem \ref{stable}}

We start this section with the following simple remark.
\begin{remark}\label{simple}
Let $(u,v)$ be a radially symmetric solution of \eq{le}. Then
$$
\lim_{r\rightarrow \infty} u(r)=\lim_{r\rightarrow \infty} v(r)=0.
$$
\end{remark}
To see this, we first note that $(u,v)$ satisfies
\neweq{uvsat}
\left\{\begin{aligned}
-&(r^{N-1}u')'= r^{N-1}v^p & \quad\mbox{ for all } r\geq 0,\\
-&(r^{N-1}v')'= r^{N-1}u^q & \quad\mbox{ for all } r\geq 0.
\end{aligned}\right.
\endeq
This implies that $r\longmapsto r^{N-1}u'(r)$ and $r\longmapsto r^{N-1}v'(r)$ are decreasing on $[0,\infty)$ and so $u',v'\leq 0$ in $[0,\infty)$.
Thus, $u$ and $v$ are decreasing in $[0,\infty)$. Hence,  there exist
$$
\ell_1:=\lim_{r\rightarrow \infty }u(r)\in [0,\infty)\,,\quad \ell_2:=\lim_{r\rightarrow \infty }v(r)\in [0,\infty),
$$
and $u\geq \ell_1$, $v\geq \ell_2$ in $[0,\infty)$.

If $\ell_2>0$, then, the first equation in \eq{uvsat} implies
$$
-(r^{N-1}u')'\geq  C r^{N-1}  \quad\mbox{ for all } r\geq 0,
$$
where $C=\ell^p_2>0$. Integrating twice over $[0,r]$ in the above inequality we deduce
$$
-u(r)+u(0)\geq \frac{C}{2N}r^2\rightarrow \infty\quad\mbox{ as }r\rightarrow \infty,
$$
contradiction. Thus, $\ell_2=0$ and similarly $\ell_1=0$ which proves our claim.

\medskip

Assume $(p,q)$ satisfies \eq{lecv}. Then by Proposition \ref{stabsing}, the singular solution $(u_s,v_s)$ is stable in $\R^N\setminus\{0\}$. 

Theorem \ref{stable}(i) follows from the proposition below.
\begin{prop}\label{pcrit}
Assume $(p,q)$ satisfies \eq{lecv}. Then, for any radially symmetric solution $(u,v)$ of \eq{le} we have
\neweq{ineq}
u<u_s\quad\mbox{ and }\quad v<v_s\quad\mbox{ in }\;\R^N\setminus\{0\}.
\endeq
\end{prop}
\begin{proof}
Assume by contradiction that there exists a radially symmetric solution $(u,v)$ of \eq{le} for which \eq{ineq} fails to hold and set
$$
U=u_s-u\;,\quad V=v_s-v.
$$
Since \eq{ineq} is not fulfilled, $U'$ and $V'$ must change sign in $(0,\infty)$. Indeed, otherwise $U'<0$ or $V'<0$ in $(0,\infty)$ which implies (since $U(\infty)=V(\infty)=0$) that $u_s\geq u$ or $v_s\geq v$ in $(0,\infty)$. Now, the maximum principle yields $u_s\geq u$ and $v_s\geq v$ in $(0,\infty)$ and this contradicts our assumption.

Let $r_1>0$ (resp. $r_2>0$) be the first zero of $U'$ (resp. $V'$). Thus
$$
U'<0 \mbox{ in }(0,r_1) ,\; U'(r_1)=0,\quad V'<0\mbox{ in }(0,r_2),\; V'(r_2)=0.
$$
Without losing the generality, we may assume $r_2\geq r_1$. Set next
$$
r_3:=\inf\{r>0:V(r)<0\}\in (0,\infty]
$$
and we claim that $r_3<r_1$. If $r_3\geq r_1$ then $V>0$ in $(0,r_1)$ which means
\neweq{vv1}
v<v_s\quad\mbox{ in }(0,r_1).
\endeq
Integrating in \eq{le} and using \eq{vv1} we find
$$
(r^{N-1}u')'=-r^{N-1}v^p>-r^{N-1}v_s^p=(r^{N-1}u'_s)'\quad\mbox{ in }(0,r_1).
$$
Integrating the above inequality over $[0,r_1]$ we find
$u'(r_1)>u'_s(r_1)$ which contradicts $U'(r_1)=0$. Hence $r_3\in (0,r_1)$. Similarly we define
$$
r_4:=\inf\{r>0:U(r)<0\}\in (0,\infty]
$$
and as before we deduce $r_4\in (0,r_2)$. In fact, we show that $r_4\leq  r_1$. Assuming the contrary, that is, $r_4> r_1$, we find $r_1< r_4<r_2$. Further, since $V'<0$ in $(0,r_2)$ we deduce $V(r)<V(r_3)=0$ for all $r\in (r_3,r_2)$ so $v_s<v$ in $(r_3,r_2)$.
Therefore,
$$
(r^{N-1}u')'=-r^{N-1}v^p<-r^{N-1}v_s^p=(r^{N-1}u'_s)'\quad\mbox{ in }(r_3,r_2).
$$
Integrating over $[r_1,r]$, $r_1<r<r_2$, and using $U'(r_1)=0$ we obtain $u'(r)<u_s'(r)$ for all $r\in (r_1,r_2)$. This means that $U$ is increasing in $(r_1,r_2)$. In particular, $U(r_1)< U(r_4)=0$. On the other hand, from the definition of $r_4$ we have $U(r_1)>0$, contradiction. We have thus obtained $r_3<r_1$, $r_4\leq r_1\leq r_2$ which yield
\neweq{vv2}
U(r_1)\leq 0,\; U'(r_1)=0,\;V(r_1)<0,\;V'(r_1)\leq 0.
\endeq
Next, let $(\phi,\psi)$ be defined by \eq{fipsi} and recall that $(\phi,\psi)$ solves the linearized equation \eq{lec1} in $\rmc$. Also, since $p\ge q\ge1$, $(U,V)$ satisfies
\neweq{lec7}
\left\{\begin{aligned}
-&\Delta U\leq pv_s^{p-1}V\\
-&\Delta V\leq qu_s^{q-1}U
\end{aligned}\right. \quad\mbox{ in }\rmc.
\endeq
We multiply the equations in \eq{lec1} by $V$ and $U$, and the two equations in \eq{lec7} by $\psi$ and $\phi$ respectively. Integrating over $B_r$, $r>0$, we find
$$
\int_{B_r}(-\Delta U)\psi\leq \int_{B_r}(-\Delta\phi)V\quad\mbox{ and }\quad \int_{B_r}(-\Delta V)\phi\leq \int_{B_r}(-\Delta \psi)U.
$$
Adding the above inequalities we deduce
$$
\int_{B_r} \Big(V\Delta \phi -\phi\Delta V\Big)+
\int_{B_r} \Big(U\Delta\psi -\psi\Delta U\Big)\leq 0\quad\mbox{ for all }r>0,
$$
that is,
$$
\int_{\partial B_r} \Big(V\frac{\partial \phi}{\partial\nu}-\phi\frac{\partial V}{\partial \nu}\Big)+
\int_{\partial B_r} \Big(U\frac{\partial \psi}{\partial\nu}-\psi\frac{\partial U}{\partial \nu}\Big)\leq 0\quad\mbox{ for all }r>0.
$$
Since $U,V,\phi,\psi$ are radially symmetric, this yields
\neweq{vv3}
V\phi'-\phi V'+U\psi'-\psi U'\leq 0\quad\mbox{ in }(0,\infty).
\endeq
Now, let us remark that $\phi,\psi>0$ and $\phi',\psi'<0$ in $(0,\infty)$.
Combining this fact with \eq{vv2} we deduce that \eq{vv3} does not hold ar $r=r_1$, a contradiction. Hence $u<u_s$ and $v<v_s$ in $\rmc$.
\end{proof}

Assume next that \eq{lecv} fails to hold. We establish first the following result.
\begin{prop}\label{louis}
Assume $(p,q)$ does not satisfy \eq{lecv}. Then, for any stable solution $(u,v)$ of \eq{le} we have
$$
u<u_s\quad\mbox{ and }\quad v<v_s\quad \mbox{ in  }\;\R^N\setminus\{0\}.
$$
\end{prop}
\begin{proof} Assume by contradiction that $u-u_s$ changes sign in $\rmc$. Then $v-v_{s}$ also changes sign in $\rmc$ for otherwise $v-v_s\leq 0$ in $\rmc$ implies
$$
-\Delta(u-u_s)=v^p-v_s^p\leq 0\quad\mbox{ in }\;\rmc.
$$
Also $u-u_s<0$ in a neighborhood of the origin and by Remark \ref{simple} we have $u(x)-u_s(x)\to0$ as $|x|\ri\infty$. By the maximum principle, we deduce $u-u_s\leq 0$ in $\rmc$ which contradicts our assumption.

Hence $u-u_s$ and $v-v_s$ change sign on $(0,\infty)$. Denote by $r_1$ (resp. $r_2$) the first sign-changing zero of $u-u_s$ (resp. $v-v_s$). From Corollary \ref{lecbalcor}, $u-u_s$ (resp. $v-v_s$) cannot be zero in a whole neighborhood of $r_1$ (resp. $r_2$). Without losing generality, we may assume that $r_1\leq r_2$.

We claim that $u-u_s$ has a second sign-changing point $r_3>r_1$. Indeed, otherwise $u-u_s\geq 0$ in $\R^N\setminus B_{r_1}$ which by the maximum principle implies that $v-v_s\geq 0$ in $\R^N\setminus B_{r_2}$. Therefore, $u\geq u_s$, $v\geq v_s$ in $\R^N\setminus B_{r_2}$ which implies that $(u_s,v_s)$ is a stable solution of \eq{le} in $\R^N\setminus B_{r_2}$ and thus, contradicts Proposition \ref{stabsing}. Hence, there exists $r_3>r_1$ a second sign-changing point of $u-u_s$. Further, we must have $r_3\geq r_2$ for otherwise $r_1<r_3< r_2$. Then $u(r_3)=u_s(r_3)$ and $v(r_3)<v_s(r_3)$ which by Corollary \ref{lecbalcor} yields $u<u_s$, $v<v_s$ in $B_{r_3}\setminus\{0\}$. But this is impossible since $u(r_1)=u_s(r_1)$. Thus, $r_3\geq r_2$.

We next claim that $v-v_s$ has a second sign-changing point $r_4>r_2$. As before, if this is not true, then $v-v_s\geq 0$ in $\R^N\setminus B_{r_2}$ and by the maximum principle we find $u-u_s\geq 0$ in $\R^N\setminus B_{r_3}$. Then $u\geq u_s$, $v\geq v_s$ in $\R^N\setminus B_{r_3}$, so $(u_s,v_s)$ is stable in $\R^N\setminus B_{r_3}$ which contradicts Proposition \ref{stabsing}.

We show next that $r_4\geq r_3$. Assuming the contrary we have $r_2<r_4<r_3$. At this stage, two cases may occur:

\noindent{\sc Case 1:} $v\leq v_s$ in $(r_4,r_3)$. Remark that $u(r_3)=u_s(r_3)$ and $v(r_3)\leq v_s(r_3)$. By Corollary \ref{lecbalcor} we deduce $u<u_s$ in $B_{r_3}$ which is impossible since $u(r_1)=u_s(r_1)$.

\noindent{\sc Case 2:} $v- v_s$ has a third sign-changing point  $\rho\in (r_4,r_3)$. Then $v-v_s>0$ on $(r_2,r_4)$ and $v-v_s<0$ on $(r_4,\rho)$. On the other hand,
$$-\Delta (v-v_s)=u^q-u_s^q\geq 0 \quad\mbox{ in }B_\rho\setminus \overline B_{r_4}
$$
and $v-v_s=0$ on $\partial (B_\rho\setminus B_{r_4})$. The maximum principle yields $v-v_s>0$ on $(r_4,\rho)$, a contradiction.
We have proved that $r_4\geq r_3$.

We claim that $u-u_s$ has a third sign-changing point $r_5>r_3$. Indeed, if this is not true, then $u-u_s\leq 0$ in $\R^N\setminus B_{r_3}$ and by the maximum principle we have $v-v_s\leq 0$ in $\R^N\setminus B_{r_4}$. Hence $u\leq u_s$, $v\leq v_s$ in $\R^N\setminus B_{r_4}$ which combined with Corollary \ref{lecbalcor} produces $u<u_s$, $v<v_s$ in $B_{r_4}$. This is clearly impossible since $u(r_1)=u_s(r_1)$.
Hence, $u-u_s$ has a third sign-changing point $r_5>r_3$.

If $r_5\leq r_4$ then
$$-\Delta (u-u_s)=v^p-v_s^p\geq 0 \quad\mbox{ in }B_{r_5}\setminus \overline B_{r_3}
$$
and $u-u_s=0$ on $\partial (B_{r_5}\setminus B_{r_3})$. By the maximum principle we infer that $u-u_s\geq 0$ in
$B_{r_5}\setminus B_{r_3}$ which implies $u-u_s\geq 0$ in
$B_{r_5}\setminus B_{r_1}$. This contradicts the fact that $r_3\in (r_1,r_5)$ is a sign-changing point of $u-u_s$.

If $r_5> r_4$ then $u(r_4)\leq u_s(r_4)$ and $v(r_4)=v_s(r_4)$. By Corollary \ref{lecbalcor} we deduce $u<u_s$, $v<v_s$ in $B_{r_4}$ which is again a contradiction.
\end{proof}

We are now ready to complete the proof of Theorem \ref{stable}(ii). We adapt an idea introduced in \cite{ddm}. Assume there exists a positive stable radially symmetric solution $(u,v)$ of \eq{le} and set
$$
M_1=\sup_{r\in (0,\infty)}\frac{u(r)}{u_s(r)}\,,\quad
M_2=\sup_{r\in (0,\infty)}\frac{v(r)}{v_s(r)}.
$$
By Proposition \ref{louis} we have $M_1,M_2\leq 1$.
Since $\lim_{r\rightarrow\infty}u(r)=0$, $u$ coincides with the Newtonian potential of $v^p$. Hence
$$
\begin{aligned}
u(x)&=c_N\int_{\R^N} |x-y|^{2-N}v^p(y)dy\\
&\leq M_2^p\left\{
c_N\int_{\R^N} |x-y|^{2-N}v_s^p(y)dy\right\}=M_2^pu_s(x).
\end{aligned}
$$
Thus, $M_1\leq M_2^p$ and similarly $M_2\leq M_1^q$. It follows that $M_1\leq M_1^{pq}$. So, since $pq>1$ we have either $M_1=0$ or $M_1=1$. If $M_1=0$ then $u\equiv 0$ and this yields $v\equiv 0$ which is impossible. Therefore $M_1=1$ and similarly $M_2=1$, i.e.
$$
\sup_{r\in (0,\infty)}\frac{u(r)}{u_s(r)}=\sup_{r\in (0,\infty)}\frac{v(r)}{v_s(r)}=1.
$$
By the strong maximum principle, $(u,v)$ cannot touch $(u_{s},v_{s})$, so there exists a sequence $\{R_k\}$ converging to $+\infty$ such that
\neweq{conv}
\lim_{k\rightarrow\infty}\frac{u(R_k)}{u_s(R_k)}=1.
\endeq
Define
$$
u_k(r)=R^\a_ku(R_kr)\,,\quad v_k(r)=R^\b_kv(R_kr)\quad r\geq 0.
$$
By scale invariance we have
\neweq{scale}
0<u_k<u_s\,,\quad 0<v_k<v_s\quad\mbox{ in } \rmc
\endeq
and $(u_k, v_k)$ solves the Lane-Emden system \eq{le} in $\rmc$. By elliptic regularity, $\{(u_k,v_k)\}$ converges uniformly in $C^2_{loc}(\rmc)$ to a solution $(\widetilde u,\widetilde v)$ of \eq{le} which, in view of \eq{scale}, also satisfies
$$
0\leq \widetilde u\leq u_s\,,\quad 0\leq\widetilde v \leq v_s\quad\mbox{ in } \rmc.
$$
Let us remark that by \eq{conv} we have
$$
\widetilde u(1)=\lim_{k\ri\infty}u_k(1)=\lim_{k\ri\infty}R_k^au(R_k)=
\lim_{k\ri\infty}R_k^au_s(R_k)=u_s(1).
$$
On the other hand,
$$
\left\{\begin{aligned}
-&\Delta (\widetilde u-u_s)=\widetilde v^p-v_s^p\leq 0 \quad\mbox{ in }\rmc,\\
&\lim_{|x|\ri 0} (\widetilde u-u_s)\leq 0\,, \lim_{|x|\ri\infty} (\widetilde u-u_s) \leq 0.
\end{aligned}\right.
$$
By the strong maximum principle we deduce that $\widetilde u\equiv u_s$ in $\rmc$. This is impossible, since $\widetilde u$ is a stable solution by construction while $u_{s}$ is unstable when \eq{lecv} fails.

\section*{Acknowledgement}
The first named author was partially supported by the Ecos-Conicyt grant C09E06 between France and Chile. The second named author wishes to thank J. Wei and the math department at the Chinese University of Hong Kong, where part of this work was completed. The research of the third named author has been supported by a Ulysses project between France and Ireland.

%
%
%
%
%
%
%
%
%

\bibliographystyle{amsalpha}

\end{document}